\documentclass[a4paper,12pt]{article}
\usepackage{amsmath}
\usepackage{amsthm}
\usepackage{amssymb}
\usepackage[latin1]{inputenc}
\usepackage{color}
\usepackage[english]{babel}
\usepackage{hyperref}

\theoremstyle{plain}
\newtheorem{thm}{Theorem}[section]
\newtheorem{prop}[thm]{Proposition}
\newtheorem{lem}[thm]{Lemma}
\newtheorem{cor}[thm]{Corollary}

\newtheorem{fact}[thm]{Fact}

\theoremstyle{definition}
\newtheorem{defn}[thm]{Definition}
\theoremstyle{remark}

\newtheorem{rem}[thm]{Remark}

\DeclareMathOperator{\Ric}{Ric}

\DeclareMathOperator{\Td}{Td}
\DeclareMathOperator{\diag}{diag}

\begin{document}
\title{Big and nef classes, Futaki Invariant and  resolutions of cubic threefolds }
\date{\today}

\author{Claudio Arezzo\footnote{ICTP Trieste and Università di Parma, arezzo@ictp.it}, Alberto Della Vedova\footnote{Università di Milano - Bicocca, alberto.dellavedova@unimib.it}}

\maketitle

\begin{abstract} 
\noindent In this note we revisit and extend few classical and recent results on the definition and use of the Futaki invariant in connection with the
existence problem for K\"ahler constant scalar curvature metrics on polarized algebraic manifolds, especially in the case of resolution of singularities. The general inspiration behind this
work is no doubt the beautiful paper by Ding and Tian \cite{DinTia1992} which contains the germs of a huge amount of the successive developments in this 
fundamental problem, and it is a great pleasure to dedicate this to Professor G. Tian on the occasion of his birthday!
\end{abstract}


\section{Introduction}

Let $X$ be a normal projective variety of dimension $n$, let $L$ be an ample line bundle on $X$, and let be fixed a $\mathbf C^*$ action on $X$ together with a linearization to $L$, that is a lifting of the given action on $X$ to an action on $L$ which is linear among the fibers.
Up to replace $L$ by some sufficiently large positive power $L^m$ (always possible for our purposes), one can suppose with no loss that $X$ is a subvariety of some complex projective space $\mathbf {CP}^d$, the line bundle $L$ is the restriction to $X$ of the hyperplane bundle, and the $\mathbf C^*$-action is induced by some one-parameter subgroup of $SL(d+1,\mathbf C)$ acting linearly on $\mathbf {CP}^d$ and leaving $X$ invariant.

Associated with these data there is a numerical invariant $F(X,L)$, named after Futaki, who introduced it as an obstruction to the existence of K\"ahler-Einstein metrics on Fano manifolds \cite{Fut1983,Fut1988}. Since then it has been widely generalized \cite{Cal1985,DinTia1992,Tia1997,Don2002,Ban2006}.
A crucial step towards a definition of stability for Fano manifolds was the extension of Futaki invariant to singular varieties. This was done by Ding-Tian, who defined a Futaki invariant for $\mathbf Q$-Fano varieties \cite{DinTia1992,Tia1997}.
Later, Donaldson defined a Futaki inavariant for polarized varieties in purely algebraic terms \cite{Don2002}. As noticed in \cite{tiCheeger}, the equivalence of all these extensions follows by results of Paul-Tian \cite{pt}.

Furthermore, the concept of Futaki invariant has been conveniently extended to the case when instead of a polarization - that is an ample line bundle - on $X$, one is given a line bundle that is just big and nef \cite{AreDelLaN2012, AreDelLaN2009}.

This last extension is of particular relevance when looking at the problem of degenerating the K\"ahler classes of canonical metrics towards the boundary of the K\"ahler cone, hence looking at possible convergence of such metrics towards singular ones.

In fact the above idea can be reversed in the hope that the existence of a singular cscK metric in a big and nef class would provide a good starting point for some deformation argument to get also smooth ones in the interior of the K\"ahler cone nearby the singular one. This turned out to be a successful strategy in a number of important situations, such as blow-ups of smooth points \cite{ap,ap2,sz,sz2}, blow-ups 
of smooth submanifolds \cite{szy}, smoothings of isolated singularities \cite{br,sp} and resolutions of isolated quotient singularities \cite{ArezzoLenaMazzieri2015,AreDelMaz2018,ADVLM}.

Besides some general observations of possible intrinsic interest, the situation studied in this note is the following:

\begin{itemize}
\item
the singular set $S$ of $X$ is finite (so that each point of S is fixed by the $\mathbf C^*$-action);
\item
$\pi : M \to X$ is an equivariant (log) resolution of singularities, i.e. $\pi$ restricts to a biholomorphism from $M \setminus \pi^{-1}(S)$ to $X \setminus S$, and for all $p\in S$ the (reduced) exceptional divisor $E_p=\pi^{-1}(p)$ is simple normal crossing;
\item
given, $p\in S$ and a collection of numbers $b_p >0$, we look at the ample line bundle $L_r = \pi^*L^r \otimes \mathcal O({\textstyle -\sum_{p \in S}} b_pE_p)$, for $r$ sufficiently large.
\end{itemize}

Our main results, Theorem \ref{thm::mainresultfutaki} and Corollary \ref{cor::mainresultintersection}, provide general formulae relating the Futaki invariant of $(X,L)$, Futaki 
of $(M,L_r)$, $b_p$, the behaviour of a potential for the $\mathbf C^*$-action at the singular points  and intersection numbers of $M$.

This results extends all known instances where a similar problem has been attacked (blow-ups at smooth points and resolutions of isolated quotient singularities in the above mentioned works), and provides many families of examples of new $K$-unstable polarized manifolds, even as resolutions of $K$-polystable normal varieties. 

Two comments are in order:
\begin{enumerate}
\item
The assumption on the normality of $X$ is not always necessary for our analysis. Yet, being the final motivation the (non-)existence of cscK metrics, we might as well assume it right away, thanks to \cite{lx};
\item
we just recall the reader that $K$-instability is indeed an obstruction to the existence of cscK metrics thanks to \cite[Theorem 1]{Don2005}.
\end{enumerate}

We end this note with the discussion of few explicit examples. Of course we need to go in dimension at least three to find non-quotient isolated singularities. In particular the case of cubic threefolds is discussed in Section \ref{sec::33folds}.
Thanks to Allcock \cite{All2003} and Liu-Xu \cite{LiuXu2017}, as recalled in Theorem \ref{tutticazzi}, $K$-polystable cubic threefolds are now classified, and for example among them it appears the zero locus $X$ of 
$$F_\Delta = x_0x_1x_2 + x_3^3 + x_4^3$$
which has three $D_4$ singularities, and continuous families of automorphisms.

Now consider a resolution $\pi : M \to X$, and let $E_j, j=1,2,3$, be the exceptional divisors.
Chosen integers $b_j>0$, consider the line bundle $$L_r = \pi^*L^r \otimes \mathcal O(-\sum_{j=0}^2 b_jE_j),$$ which is ample for all $r$ sufficiently large.

By applying our general computation of the Futaki invariant, we will show (see  Proposition \ref{FDelta}) that 
any polarized resolution $(M,L_r)$ of the cubic threefold $F_\Delta=0$ is K-unstable for $r$ sufficiently large as soon as the intersection numbers $K_M \cdot (b_0E_0)^{2}, K_M \cdot (b_1E_1)^{2},K_M \cdot (b_2E_2)^{2}$ are not all the same.

The same strategy can be applied for other examples as discussed in Section \ref{sec::33folds}.

\subsection*{Acknowledgments} Both authors at different times and places have benfited from hundreds of conversations with Prof. G. Tian on topics related to the ones studied in this note. It is a great pleasure to dedicate this paper to him, with our best wishes for his birthday!

\section{Futaki invariant}

In this section we give an account of the extension of the Futaki invariant to big and nef classes developed in \cite{AreDelLaN2009,AreDelLaN2012}.

Recall that a line bundle $B$ on a projective variety $X$ of dimension $n$ is said to be big when it has positive volume, the latter being the limit of $\dim H^0(X,B^k)/k^n$ as $k \to +\infty$.
On the other hand, $B$ is said to be nef if, for any irreducible curve $\Sigma \subset X$, the restriction of $B$ to $\Sigma$ has non-negative degree.
By Kleiman's theorem, nefness is the closure of ampleness condition, meaning that $B$ turns out to be nef if and only if for any ample line bundle $A$ there is $k>0$ such that $B^k \otimes A$ is ample.
On a smooth projective manifold, a line bundle is big and nef if and only if its first Chern class lies at the boundary of the K\"ahler cone and has positive self-intersection.

\begin{defn}\label{defn::Futaki}
Let $X$ be a normal projective variety endowed with a $\mathbf C^*$-action and let $B$ a big and nef line bundle on $M$.
Choose a linearization on $B$ and for all $k\geq0$ consider the virtual $\mathbf C^*$-representation $H_k = \sum_{q \geq 0} (-1)^q H^q(X,B^k)$.
Let $\chi(X,B^k)=\dim(H_k)$ be the Euler characteristic of $B^k$ and let $w(X,B^k)$ be the trace of the infinitesimal generator of the representation $H_k$.
For $k \to \infty$ we have an asymptotic expansion
\begin{equation}\label{eq::expdefnfutaki}
\frac{w(X,B^k)}{\chi(X,B^k)} = F_0k + F_1 + O(k^{-1}),
\end{equation}
and the Futaki invariant $F(X,B)$ of the given $\mathbf C^*$-action on $X$ is defined to be the constant term $F_1$ of expansion above.
\end{defn}

A few comments on this definition are in order.

Firstly, note that given $X$ acted on by $\mathbf C^*$ and $B$ as in the definition, one can always find a linearization of the action to $B$ \cite[Theorem 7.2]{Dol2003}.
Actually, in order to do this, perhaps one should replace $B$ with a $\mathbf C^*$-invariant line bundle $B'$ isomorphic to $B$. Since this replacement has no effect for our purposes, from now on we implicitly assume that any line bundle on $X$ is endowed with a linearization of the given $\mathbf C^*$-action on $X$.    
On the other hand, $F(X,B)$ does not depend on the chosen linearization, whereas the representation $H_k$ and the weight $w(X,B^k)$ do depend on it.
In fact, one can check that altering the linearization has the effect of adding $\lambda k \chi(X,B^k)$ to the weight $w(X,B^k)$ for some $\lambda \neq 0$, so that $F_1$ in expansion \eqref{eq::expdefnfutaki} stay unchanged.

Secondly, note that whenever $B$ is ample, $B^k$ has no higher cohomology for $k$ positive and sufficiently large. Therefore $H_k$ is a genuine representation of $\mathbf C^*$, and finally one recovers the Donaldson's definition of Futaki invariant \cite[Subsection 2.1]{Don2002}.

Thirdly, in the general case one has $\lim_{k\to\infty} k^{-n}\dim H^0(X,B^k) >0$ by definition of bigness, and $\dim H^q(X,B^k) = O(k^{n-q})$ as a consequence of nefness \cite[Theorem 1.4.40]{Laz2004}.
Hence, even in the more general case, in order to compute $F(X,B)$, one has to consider cohomology groups of $B$ up to order $q=1$.

Finally, note that for any fixed $m>0$ replacing $k$ with $mk$ in \eqref{eq::expdefnfutaki} yields the identity
\begin{equation}\label{eq::degzerohom}
F(X,B^m) = F(X,B).
\end{equation}

One advantage of definition above is that it extends the classical Futaki invariant continuously up to points of the boundary of the ample cone having non-zero volume.
More specifically, it holds the following

\begin{prop}\label{prop::contFutaki}
Let  $X$ be a normal projective variety endowed with a $\mathbf C^*$-action.
For all line bundles $B$ big and nef, and $F$ invariantly effective, as $r \to \infty$ one has
\begin{equation}
F(X,B^r \otimes F) = F(X,B) + O(1/r).
\end{equation} 
\end{prop}
\begin{rem}\label{rem:defnequiveffect}
By invariantly effective line bundle, we mean a line bundle $F$ such that some positive power $F^m$ posses a $\mathbf C^*$-invariant non-zero section.
For example, any ample line bundle on $X$ is invariantly effective.
Another example is the line bundle $\mathcal O(-D)$ associated with a $\mathbf C^*$-invariant hypersurface $D \subset X$. 
In particular, the line bundle associated with an exceptional divisor of a blow-up is invariantly effective.
\end{rem}
\begin{proof}[Proof of theorem \ref{prop::contFutaki}]
For ease of notation let $B_r = B^r \otimes F$.
Note that by \eqref{eq::degzerohom} one can replace $B_r$ with an arbitrary large power without altering $F(X,B_r)$.
Therefore we can assume that there is an invariant section of $F$, and let $D \subset X$ be its null locus.
Multiplication by $k$-th power of the chosen section gives an equivariant sequence of sheaves on $X$
\begin{equation}
0 \to B^{rk} \to B_r^k \to \left.B_r^k \right|_{kD} \to 0
\end{equation}
which induces a sequence of (virtual) representation of $\mathbf C^*$, whence one has
\begin{equation}
\chi(X,B_r^k)
= \chi(X,B^{rk}) + k\chi(D,\left.B_r^k \right|_D).
\end{equation}
and
\begin{equation}
w(X,B_r^k)
= w(X,B^{rk}) + kw(D,\left.B_r^k \right|_D).
\end{equation}
Note that by bigness and nefness of $B$ and by asymptotic Riemann-Roch theorem there is a polynomial $q(t)=q_0t^n + \dots + q_n$ with $q_0>0$ such that $\chi(X,B^{rk})=q(rk)$ \cite[Theorems 1.1.24 and 2.2.16]{Laz2004}.
Similarly, $w(X,B^{rk})=p(rk)$ for some polynomial $p(t) = p_0t^{n+1} + \dots + p_{n+1}$.
For the same reasons, since $D$ has dimension $n-1$, the Euler characteristic $\chi(D,\left.B_r^k \right|_D)=\tilde q(r,k)$ is a polynomial of the form $\tilde q_0(r)k^{n-1}+\dots+\tilde q_{n-1}$ with $\tilde q_i(r)$ which are polynomials of degree at most $n-1-i$ and $\tilde q_0(r)>0$ for $r>0$.
A similar situation stand for the total weight $w(D,\left.B_r^k \right|_D)=\tilde p(r,k)$ with all degrees raised by one.
The upshot is that
\begin{equation}
\frac{w(X,B_r^k)}{\chi(X,B_r^k)} = \frac{p(rk) + k\tilde p(r,k)}{q(rk) + k \tilde q(r,k)}.
\end{equation}
Expanding the polynomials, by definition of Futaki invariant one finds
\begin{equation}
F(X,B_r) = \frac{p_1 + \tilde p_1(r)r^{-n}}{q_0 + \tilde q_0(r)r^{-n}} - \frac{\left(p_0 + \tilde p_0(r)/r^{n+1}\right)\left(q_1 + \tilde q_1(r)/r^{n-1}\right)}{\left(q_0 + \tilde q_0(r)r^{-n}\right)^2}.
\end{equation}
At this point, note that $F(X,B) = p_1/q_0 - p_0q_1/q_0^2$.
On the other hand, by discussion above we know that $\tilde p_i(r)/r^{n+1-i}$ and $\tilde q_i(r)/r^{n-i}$ are $O(1/r)$ for large $r$.
Therefore $F(X,B_r) = F(X,B) + O(1/r)$ as $r \to \infty$, which is the thesis.
\end{proof}

Thanks to definition \ref{defn::Futaki}, one can equally work on a singular projective variety endowed with an ample line bundle, or on a smooth variety endowed with a big and nef line bundle, as shown by the following
\begin{prop}\label{cor::Futakipullbackres}
Let $X$ be a normal variety endowed with a $\mathbf C^*$-action and an ample line bundle $L$.
Let $\pi: M \to X$ be an equivariant resolution of singularities.
One has
\begin{equation*}
F(M, \pi^*L) = F(X,L).
\end{equation*}
\end{prop}
\begin{proof}
Note that $\pi^*L$ is big and nef on $M$, so that l.h.s of the identity in the statement makes sense.
Now observe that there is an eqivariant sequence of sheaves on $X$
\begin{equation}
0 \to \mathcal O_X \to \pi_* \mathcal O_M \to \eta \to 0
\end{equation}
where the support of $\eta$ has co-dimension at least two.
Indeed, the support of $\eta$ is contained in the singular locus of $X$, and the latter has co-dimension at least two by normality assumption.
After twisting by $L^k$, by projection formula one then sees that
\begin{equation*}
w(M,\pi^*L^k) = w(X,L^k) + O(k^{n-1}), \qquad \chi(M,\pi^*L^k) = \chi(M,L^k) + O(k^{n-2}),
\end{equation*}
whence the thesis follows by definition of Futaki invariant.
\end{proof}

Combining propositions \ref{prop::contFutaki} and \ref{cor::Futakipullbackres} one readily gets the following
\begin{cor}\label{cor::Futakiblowup}
In the situation of proposition \ref{cor::Futakipullbackres}, let $F$ be an invariantly effective divisor on $M$ (cfr. remark \ref{rem:defnequiveffect}). For $r \to \infty$ one has
\begin{equation*}
F(M,\pi^*L^r \otimes F) = F(X,L) + O(1/r).
\end{equation*}
\end{cor}

In the next section, we shall make more explicit the error term $O(1/r)$, at least when the singularities of $X$ are not too bad.

%

\section{Resolutions of isolated singularities}

In this section we consider the Futaki invariant of adiabatic polarizations (i.e. making small the volume of exceptional divisors) on resolution of isolated singularities.

As above, consider a normal projective variety $X$ of dimension $n$ endowed with a $\mathbf C^*$-action, and let $L$ be an ample line bundle on $X$. In this section we make the additional assumptions that $X$ is $\mathbf Q$-Gorenstein with at most isolated singularities \cite{Ish1987}.
This means that the singular set $S \subset X$ is finite and each $p \in S$ is a fixed point for the $\mathbf C^*$-action.
Moreover, some tensor power of the canonical bundle of the smooth locus $X \setminus S$ extends to a line bundle on $X$.
Note that this makes the canonical bundle $K_X$ of $X$ a $\mathbf Q$-line bundle, meaning that $K_X^m$ is a genuine line bundle for some integer $m>0$.

Now consider an equivariant (log) resolution of singularities $\pi : M \to X$. By definition, $\pi$ restricts to a biholomorphism from $M \setminus \pi^{-1}(S)$ to $X \setminus S$, and for all $p\in S$ the (reduced) exceptional divisor $E_p=\pi^{-1}(p)$ is simple normal crossing.
 
Given a positive constant $b_p$ for each $p \in S$, there is $r$ sufficiently large such that the line bundle
\begin{equation}
L_r = \pi^*L^r \otimes \mathcal O({\textstyle -\sum_{p \in S}} b_pE_p)
\end{equation} is ample on $M$.
Moreover, $\pi^*L$ is big and nef, and each line bundle $\mathcal O(-E_p)$ is invariantly effective (cfr. remark \ref{rem:defnequiveffect}) for $E_p$ is invariant.
Note that corollary \ref{cor::Futakiblowup} applies, so that for large $r$ it holds
\begin{equation}
F(M,L_r) = F(X,L) + O(1/r).
\end{equation}

In order to make somehow more explicit the error term, consider the virtual representation $H_k=\sum_{q \geq 0} (-1)^q H^q(M,L_r^k)$.
Since $M$ is smooth, at least for $t \in \mathbf R$ sufficiently small, the character $\chi_{H_k}$ of such representation satisfies \cite[Theorem 8.2]{BerGetVer2004}
\begin{equation}\label{eq::RRequiv}
\chi_{H_k}(e^{it}) = \int_M e^{c_1(L_r^k)}\Td(M),
\end{equation}
where $c_1(L_r^k)$ and $\Td(M)$ are equivariant characteristic classes.
To be more specific, consider the unit circle inside $\mathbf C^*$ and let $V \in \Gamma(TM)$ be the infinitesimal generator of the induced circle action on $M$.
Moreover, let $\omega_r$ be a circle-invariant K\"ahler form representing the first Chern class of $L_r$, and let $u_r \in C^\infty(M)$ be a potential for the circle action on $M$, so that
\begin{equation}
i_V \omega_r = du_r.
\end{equation}
Denoting by $\Delta_r$ the Laplace operator of the K\"ahler metric $\omega_r$, then \eqref{eq::RRequiv} reduces to
\begin{equation}\label{eq::RRequivCC}
\chi_{H_k}(e^{it}) = \int_M e^{k(\omega_r+tu_r)}\left(1+\frac{1}{2}(\Ric(\omega_r)-t\Delta_r u_r) + \dots\right),
\end{equation}
where dots stand for higher order terms that are irrelevant for our purposes, and the integral of any differential form of degree different form $2n$ is defined to be zero.

In order to determine the Futaki invariant $F(M,L_r)$, we need to consider the asymptotic behavior for large $k$ of the Euler characteristic $\chi(M,L_r^k)$ and the trace $w(M,L_r^k)$ of the infinitesimal generator of the virtual representation $H_k$.
Note that by definition of $\chi_{H_k}$ one has $\chi(M,L_r^k)=\chi_{H_k}(1)$ and $w(M,L_r^k)=\left.\frac{d\chi_{H_k}(e^{it})}{dt}\right|_{t=0}$.
Therefore formula \eqref{eq::RRequivCC} gives $w(M,L_r^k)=a(r)k^{n+1}+b(r)k^n+O(k^{n-1})$, and $\chi(M,L_r^k)=c(r)k^n + d(r)k^{n-1}+O(k^{n-2})$ where
\begin{align}
a(r) &= \int_M \frac{(\omega_r+u_r)^{n+1}}{(n+1)!} & b(r) &= \int_M \frac{(\omega_r+u_r)^n \wedge (\Ric(\omega_r) - \Delta_r u_r)}{2n!} \nonumber \\ \label{eq::coeffexpchiandw}
c(r) &= \int_M \frac{(\omega_r+u_r)^n}{n!} & d(r) &= \int_M \frac{(\omega_r+u_r)^{n-1} \wedge (\Ric(\omega_r) - \Delta_r u_r)}{2(n-1)!} 
\end{align}
are polynomial functions of $r$.
Note that $b(r)$ could be simplified a bit by showing that the summand involving $\Delta_ru_r$ vanishes.
On the other hand, $u_r$ and $\Delta_ru_r$ do not affect the value of $c(r)$ and $d(r)$.
However it will be apparent in a moment that is convenient to keep the integrands expressed as polynomials in $\omega_r+u_r$ and $\Ric(\omega_r) - \Delta_r u_r$. Indeed both of these differential forms turns out to be equivariantly closed, meaning that they are circle-invariant and belong to the kernel of the differential operator
\begin{equation}
d_V = d - i_V.
\end{equation}
Note that one has $d_V^2=0$ on the space of circle-invariant differential forms.
As a consequence $d_V$ defines a cohomology, which is sometimes called (the Cartan model of) the equivariant cohomology of $M$ with respect to the given circle action. 
The equivariant characteristic classes appearing in \eqref{eq::RRequiv} belong to this cohomology.

Apart the deep result represented by \eqref{eq::RRequiv}, we need just some basic features of equivariant cohomology.
In particular, below we repeatedly make use of the following integration by part formula, whose proof is a quite direct application of the Stokes' theorem.
\begin{lem}\label{lem::stokesformula}
For all circle invariant inhomogeneous differential forms $\alpha$, $\beta$ on $M$ one has
\begin{equation*}
\int_M d_V\alpha \wedge \beta = \int_M (\alpha_{odd}-\alpha_{even}) \wedge d_V \beta,
\end{equation*}
where $\alpha=\alpha_{even}+\alpha_{odd}$ with obvious meaning.
\end{lem}

At this point we come back to our problem of finding an asymptotic expansion for $F(M,L_r)$.
By definition \ref{defn::Futaki} of Futaki invariant one readily sees that 
\begin{equation}\label{eq:Futakipolynomial}
F(M,L_r) = b(r)/c(r) - a(r)d(r)/c(r)^2.
\end{equation}
Therefore we are lead to express most of coefficients of polynomials in \eqref{eq::coeffexpchiandw} in terms of geometric data on $X$ and $M$.
In order to do this we need to introduce more notation. 

For any exceptional divisor $E_p$ let $\xi_p \in \Omega^{1,1}(M)$ a closed form which represents the Poincaré dual and it is positive along $E_p$.
If $E_p$ is smooth, the latter requirement simply means that $\xi_p$ restricts to a K\"ahler metric on $E_p$.
In general, it means that $\int_\Sigma \gamma^*\xi_p>0$ for any non-constant holomorphic curve $\gamma:\Sigma \to M$ whose image is contained in $E_p$.

We can assume that the supports of $\xi_p$ and $\xi_q$ are disjoint whenever $p,q \in S$ are distinct.
Even more, we can assume that $\xi_p$ has support contained in a circle-invariant open set $W_p$ and that $W_p$ and $W_q$ are disjoint whenever $p,q \in S$ are distinct.
Therefore, perhaps after averaging over the circle, we can also assume that $\xi_p$ is circle-invariant.
Moreover, let $u_p$ be a potential for the vector field $V$ with respect to $\xi_p$, meaning that $i_V \xi_p = du_p$.
Note that $u_p$ is defined up to an additive constant, and that it is constant in the complement of the support of $\xi_p$. 
Therefore, by fixing the additive constant, we can assume that the support of $u_p$ is contained $W_p$.
Summarizing, for any $p \in S$ there is an equivariantly closed differential form $\xi_p+u_p$ supported inside $W_p$ such that $[\xi_p] \in H^{1,1}(M)$ is Poincaré dual to $E_p$.

We already observed in the previous section that for our purposes we can assume with no loss that $X$ is an invariant subvariety of some complex projective space $\mathbf {CP}^d$ acted on linearly by some one-parameter subgroup of $SL(d+1,\mathbf C)$, and $L$ is the restriction of the hyperplane bundle to $X$.
Therefore, if 
\begin{equation*}
\iota : X \to \mathbf{CP}^d
\end{equation*}
denotes the inclusion, then the composition $\iota \circ \pi$ is a smooth equivariant map form $M$ to $\mathbf{CP}^d$ which pulls-back the hyperplane bundle to $\pi^*L$.

Thanks to the inclusion $\iota$ we can equip $X$ (or more correctly its smooth locus $X \setminus S$) with a K\"ahler metric $\omega$ and a hamiltonian potential $u$ for the circle action induced by the unit circle of $\mathbf C^*$.
To see this, let $V_{FS} \in \Gamma(T\mathbf{CP}^d)$ be its infinitesimal generator of such circle action.
Moreover, let $\omega_{FS}$ be a circle-invariant Fubini-Study metric on $\mathbf{CP}^d$.
Now a potential $u_{FS}$ for $V_{FS}$ is a smooth function on $\mathbf{CP}^d$ satisfying $i_{V_{FS}} \omega_{FS} = du_{FS}$.
Finally we define the K\"ahler form $\omega$ and the potential $u$ as the restriction to $X$ of $\omega_{FS}$ and $u_{FS}$ respectively. We can think of $\omega + u$ as an equivariantly closed differential form on $X$.
Whereas $\omega + u$ is a genuine equivariantly closed differential form on the smooth locus of $X$, it is delicate to specify what is $\omega$ at singular points of $X$. On the other hand, it is clear that $u$ is a continuous function on $X$.
However, the pull-back $\pi^*(\omega + u)$ is smooth on $M$ since it is nothing but the pull-back of $\omega_{FS}+u_{FS}$ via the composition of $\pi$ with the inclusion $\iota$ of $X$ into $\mathbf{CP}^d$.

At this point, note that we are free to shrinking the set $W_p$ in order to assume that it is contained in $(\iota \circ \pi)^{-1}(B_p)$ for some small ball $B_p \subset \mathbf{CP}^d$ centered at $p$.
As a consequence $\pi^*(\omega + u)$ turns out to be equivariantly exact in $W_p$ since $\omega_{FS}+u_{FS}$ is equivariantly exact in $B_p$ (in fact one can check that $\omega_{FS}+u_{FS} = d_Vd^c \log(1+|z|^2)$ in affine coordinates making diagonal the circle action). 
More specifically, there is a circle-invariant function $\phi_p$ on $M$ such that
\begin{equation}\label{eq::localalphap}
\pi^*(\omega+u) = d_V d^c \phi_p \qquad \mbox{in } W_p 
\end{equation}

Given all of this, we can assume that the K\"ahler metric $\omega_r$ and the potential function $u_r$ satisfy
\begin{equation}\label{eq::expansionomegarnadur}
\omega_r + u_r = r \pi^*(\omega + u) + \sum_{p \in S} b_p (\xi_p+u_p).
\end{equation}

Finally we recall a result that will be useful in the following {\cite[p. 6]{AreDelMaz2018}}.

\begin{lem}\label{lem::equivPD}
Any equivariantly closed differential form $\alpha$ on $M$ which is exact on $W_p$  and restricts to the zero form on the exceptional divisor $E_p$ satisfies $\int_M \alpha \wedge (\xi_p + u_p) = 0$.
\end{lem}

Now we are ready to make explicit coefficients of polynomials appearing in \eqref{eq:Futakipolynomial}.
Starting with $a(r)$, note that our assumption that $\xi_p+u_p$ is supported inside $W_p$ yields
\begin{equation*}
a(r) = r^{n+1} \int_{M \setminus \bigcup_p W_p} \frac{\pi^*(\omega+u)^{n+1}}{(n+1)!}
+ \sum_{p \in S} \int_{W_p} \frac{\left( r\pi^*(\omega+u) + b_p(\xi_p+u_p) \right)^{n+1}}{(n+1)!}.
\end{equation*}
Moreover, observing that $\pi^*(\omega+u)-u(p)$ restricts to zero on $E_p$, by \eqref{eq::localalphap} and lemmata \ref{lem::stokesformula}, \ref{lem::equivPD} equation above reduces to
\begin{equation}
a(r) = a_0 r^{n+1} 
+ r\sum_{p \in S} b_p^n u(p) \int_M \frac{\xi_p^n}{n!} 
+ \sum_{p \in S} b_p^{n+1} \int_M u_p \frac{\xi_p^n}{n!},
\end{equation}
where $a_0 = \int_X u\, \omega^n/n!$ coincides with the integral on $M$ of the pull-back via $\iota \circ \pi$ of the smooth differential form $u_{FS}\omega_{FS}^n/n!$.
Similarly, for $c(r)$ one finds
\begin{equation}
c(r) = c_0 r^n + \sum_{p \in S} 
b_p^n \int_M \frac{\xi_p^n}{n!},
\end{equation}
where $c_0 = \int_X \omega^n/n!$ is the volume of the line bundle $L$ on $X$, or equivaletnly the volume of $\pi^*L$ on $M$.

Now pass to consider $b(r)$.
Arguing precisely as above we can write
\begin{multline}
b(r) = r^n \int_{M \setminus \bigcup_p W_p} \frac{\pi^*(\omega+u)^n \wedge \pi^*(\Ric(\omega) - \Delta u)}{2n!} \\
+ \sum_{p \in S} \int_{W_p} \frac{\left( r\pi^*(\omega+u) + b_p(\xi_p+u_p) \right)^n \wedge (\Ric(\omega_r) - \Delta_r u_r)}{2n!},
\end{multline}
whence, again by summing and subtracting $u(p)$ to $\pi^*(\omega+u)$ and using \eqref{eq::localalphap} and lemmata \ref{lem::stokesformula}, \ref{lem::equivPD} as before, it follows
\begin{multline}\label{eq::prelimexpbr}
b(r) = b_0 r^n
+ r \sum_{p \in S} u(p) b_p^{n-1} \int_M \frac{(\xi_p+u_p)^{n-1} \wedge (\Ric(\omega_r) - \Delta_r u_r)}{2(n-1)!} \\
+ \sum_{p \in S} b_p^n \int_M \frac{(\xi_p+u_p)^n \wedge (\Ric(\omega_r) - \Delta_r u_r)}{2n!},
\end{multline}
where $b_0 = \int_M \pi^*(\omega+u)^n \wedge (\Ric(\omega_r) - \Delta_r u_r)/(2n!)$ does not depend on $r$.
This follows by integration by parts (lemma \ref{lem::stokesformula}) and the fact that for all $r,s>0$ it holds 
\begin{equation}\label{eq:transRicciomegarands}
\Ric(\omega_r) - \Delta_r u_r = \Ric(\omega_s) - \Delta_s u_s - d_Vd^c \log(\omega_r^n/\omega_s^n).
\end{equation}
For the same reason, both integrals of formula \eqref{eq::prelimexpbr} do not depend on $r$.
In fact, the one of the first line reduces to
\begin{equation*}
\int_M \frac{(\xi_p+u_p)^{n-1} \wedge (\Ric(\omega_r) - \Delta_r u_r)}{2(n-1)!}
= \int_M \frac{\xi_p^{n-1} \wedge \Ric(\omega_r)}{2(n-1)!}.
\end{equation*}
Moreover, focusing on the second line of \eqref{eq::prelimexpbr}, let $I = \int_M (\xi_p+u_p)^n \wedge (\Ric(\omega_r) - \Delta_r u_r)/(2n!)$.
In order to find a simpler expression for it, let $B_\varepsilon \subset M$ be the pullback via $\iota \circ \pi$ of a small ball in $\mathbf{CP}^d$ of radius $\varepsilon$ and centered at $p$.
Since $\pi^*\omega$ is a K\"ahler metric on $W_p\setminus B_\varepsilon$, there one can write
\begin{equation*}
\Ric(\omega_r) - \Delta_r u_r = \pi^*(\Ric(\omega)-\Delta u) - d_Vd^c \log(\omega_r^n/\pi^*\omega^n).
\end{equation*}
Therefore, being $\xi_p+u_p$ supported in $W_p$, by Stokes' theorem it follows
\begin{multline*}
I = \int_{M\setminus B_\varepsilon} \frac{(\xi_p+u_p)^n \wedge \pi^*(\Ric(\omega) - \Delta u)}{2n!}
+ \int_{\partial B_\varepsilon} \frac{(\xi_p+u_p)^n \wedge d^c \log(\omega_r^n/\pi^*\omega^n)}{2n!} \\
+ \int_{B_\varepsilon} \frac{(\xi_p+u_p)^n \wedge (\Ric(\omega_r) - \Delta_r u_r)}{2n!}.
\end{multline*}
As we already observed after equation \eqref{eq:transRicciomegarands}, $I$ does not depend on $r$.
On the other hand, note that $d^c \log(\omega_r^n/\pi^*\omega^n)$ is smooth on $\partial B_\varepsilon$ for all $r$ and is $O(1/r)$ for large $r$. 
Similarly, $\Ric(\omega_r) - \Delta_r u_r$ is smooth on $B_\varepsilon$.
Therefore, passing to the limit $r \to \infty$ in equation above yields
\begin{equation}
I = \int_M \frac{(\xi_p+u_p)^n \wedge \pi^*(\Ric(\omega) - \Delta u)}{2n!}.
\end{equation}
Note that $\Delta u$ is a continuous function on $X$.
This can be checked after noting that $\Delta u$ equals the ratio of the restrictions to $X$ of $n L_{JV_{FS}}\omega_{FS} \wedge \omega_{FS}^{n-1}$ and $\omega_{FS}^n$.
On the other hand, note that $\pi^*(\Ric(\omega) - \Delta u)$ represents the first Chern class of the line bundle $\pi^* K_X^{-1}$.
At this point consider the shifted form $\alpha = \pi^*(\Ric(\omega) - \Delta u) + \Delta u (p)$ so that 
Therefore one can rewrite
\begin{equation}
I = - \Delta u (p) \int_M \frac{\xi_p^n}{2n!} + \int_M \frac{(\xi_p+u_p)^n \wedge \alpha}{2n!}.
\end{equation}
Since $\alpha$ vanishes on $E_p$, by lemma \ref{lem::equivPD} it follows that $I$ reduces to the first summand of equation above.
As a consequence, \eqref{eq::prelimexpbr} reduces to
\begin{equation}\label{eq::finalexpbr}
b(r) = b_0 r^n
+ r \sum_{p \in S} u(p) b_p^{n-1} \int_M \frac{\xi_p^{n-1} \wedge \Ric(\omega_r)}{2(n-1)!}
-\frac{1}{2} \sum_{p \in S} \Delta u(p) b_p^n \int_M \frac{\xi_p^n}{n!},
\end{equation}

Finally, a similar and easier argument for $d(r)$ gives the expansion
\begin{equation}
d(r) = d_0 r^{n-1} + \sum_{p \in S} b_p^{n-1} \int_M \frac{\xi_p^{n-1} \wedge \Ric(\omega_r)}{2(n-1)!},
\end{equation}
where $d_0 =\int_M \pi^*\omega^{n-1} \wedge \Ric(\omega_r)/(2(n-1)!)$ does not depend on $r$, and by asymptotic Riemann-Roch theorem it is equal to $K_X\cdot L^{n-1}(2(n-1)!)$

At this point, note that we found a geometric meaning for all coefficients appearing in polynomials
\begin{align*}
a(r) &= a_0r^{n+1}+a_nr+a_{n+1} & b(r) &= b_0r^n+b_{n-1}r+b_n \\
c(r) &=c_0r^n+c_n & d(r) &= d_0r^{n-1}+d_{n-1}.
\end{align*}
By direct calculation starting form \eqref{eq:Futakipolynomial} one finds
\begin{multline}
F(M,L_r) =  \frac{b_0}{c_0} - \frac{a_0d_0}{c_0^2} + \left(\frac{b_{n-1}}{c_0} - \frac{a_0d_{n-1}}{c_0^2}\right) r^{1-n} \\
 +\left(\frac{b_n}{c_0}+\frac{d_0}{c_0}\frac{a_0c_n-c_0a_n}{c_0^2} - \frac{c_n}{c_0} \left(\frac{b_0}{c_0} - \frac{a_0d_0}{c_0^2}\right) \right)r^{-n} + O(r^{-n-1}),
\end{multline}
as $r \to \infty$.
By proposition \ref{prop::contFutaki} we can recognize $F(X,L)$ in the leading term.
Therefore, substituting coefficients calculated above yields the following
\begin{thm}\label{thm::mainresultfutaki}
Let $\pi: M \to X$ be an equivariant log resolution of a $\mathbf Q$-Gorenstein polarized variety $(X,L)$ acted on by $\mathbf C^*$. 
Assume that the singular locus $S \subset X$ is finite and choose a rational constant $b_p>0$ for all $p \in S$.
With notation introduced above, the Futaki invariant of $L_r = \pi^*L^r \otimes \mathcal O({\textstyle -\sum_{p \in S}} b_pE_p)$ for $r \to \infty$ is given by
%
\begin{multline}\label{eq::Futakinmainthem}
F(M,L_r) 
= F(X,L) 
+ r^{1-n} \frac{n}{2}\sum_{p \in S} (u(p) - \underline u) b_p^{n-1} \frac{\int_M \xi_p^{n-1} \wedge \Ric(\omega_r)}{\int_X \omega^n} \\
- \frac{1}{2} r^{-n} \sum_{p \in S} \left(\underline s(u(p) - \underline u) + \Delta u(p) + 2F(X,L) \right) b_p^n \frac{\int_M \xi_p^n}{\int_X \omega^n} 
+ O(r^{-n-1}),
\end{multline}
where $\underline s = \frac{n}{2}\int_M \pi^*\omega^{n-1} \wedge \Ric(\omega_r)/\int_X \omega^n$ does not depend on $r$.
\end{thm}
This result should be considered as an extension of a similar result for isolated quotient singularities \cite[Theorem 2.3]{AreDelMaz2018}.
Some differences with the formula appearing there are due to a different normalization in definition of Futaki invariant.

On the other hand, note that at least the first error term in \eqref{eq::Futakinmainthem} can be expressed almost entirely in terms of intersections numbers on $M$.
Therefore we have the following
\begin{cor}\label{cor::mainresultintersection}
In the situation above, as $r \to \infty$ one has
\begin{equation*}
F(M,L_r)
= F(X,L) 
- r^{1-n} \frac{n}{2L^n}\sum_{p \in S} (u(p) - \underline u) K_M \cdot (b_pE_p)^{n-1} + O(r^{-n}).
\end{equation*}
\end{cor}
This result will be useful in order to produce several examples of K-unstable resolutions in the next section.

\section{Resolutions of semi-stable cubic threefolds}\label{sec::33folds}
In this section we show that most resolution of semi-stable cubic threefolds are K-unstable.
Here we do not need to recall the full definition of K-stability. Instead it is enough to recall that it is a GIT stability notion for polarized varieties (when no polarization is specified, it is assumed to be the anti-canonical bundle), and that the Hilbert-Mumford criterion for $K$-stability implies the following elementary
\begin{fact}
A polarized variety is K-unstable as soon as it carries a $\mathbf C^*$-action with non-zero Futaki invariant.
\end{fact}

To begin with observe that by results of Allcock \cite{All2003} and Liu-Xu \cite{LiuXu2017} we have the following clear picture of K-stability of cubic threfolds. 
\begin{thm}
\label{tutticazzi}
Let $X \subset \mathbf{CP}^4$ be a cubic threefold.
\begin{itemize}
\item $X$ is $K$-stable if and only if it is smooth or it has isolated singularities of type $A_k$ with $k \leq 4$.
\item $X$ is K-polystable with non-discrete automorphism group if and only if it is projectively equivalent to the zero locus of one of the following cubic polynomials:
\begin{equation*}
F_\Delta = x_0x_1x_2 + x_3^3 + x_4^3, \qquad F_{A,B} = Ax_2^3 + x_0x_3^2 + x_1^2x_4 - x_0x_2x_4 + Bx_1x_2x_3,
\end{equation*}
with $A$ and $B$ which are not both zero.
\end{itemize} 
\end{thm}
Resolutions of K-stable cubic threefolds have no non-trivial holomorphic vector fields.
Therefore, in order to study K-instability of their resolutions one should consider test configuration along the lines of \cite{Sto2010,Del2008}.
On the other hand, studying K-instability of resolutions of (strictly) K-polystable cubic threefolds is more direct thanks to corollary \ref{cor::mainresultintersection}.
In view of this application, observe that any K-polystable cubic threefold $X \subset \mathbf{CP}^4$ is $\mathbf Q$-Gorenstein, in that the anti-canonical bundle of the smooth locus extends to $K_X^{-1}$.
Moreover, the latter is (very) ample and the restriction $L$ of hyperplane bundle to $X$ satisfies $L^2 = K_X^{-1}$.
We consider separately the cases $F_\Delta$ and $F_{A,B}$.

\subsection{$F_\Delta$}
Let $X \subset \mathbf{CP}^4$ be the zero locus of $F_\Delta = x_0x_1x_2 + x_3^3 + x_4^3$.
As one can readily check, the singular locus of $X$ is constituted by three coordinate points 
\begin{equation}
S = \left\{ p_0=(1:0:0:0:0), p_1=(0:1:0:0:0), p_2=(0:0:1:0:0) \right\}.
\end{equation}
Each of them is a $D_4$ singularity, since $X$ is locally equivalent to $z_1^2+z_2^2+z_3^3+z_4^3$ around any $p \in S$.
Now pick $\alpha_0,\alpha_1,\alpha_2 \in \mathbf Z$ such that $\alpha_0+\alpha_1+\alpha_2=0$ and consider the diagonal action of $\mathbf C^*$ on $\mathbf{CP}^4$ induced by $\diag(t^{\alpha_1},t^{\alpha_2},t^{\alpha_3},1,1)$, where $t \in \mathbf C^*$.
Clearly $X$ is invariant with respect to this action.
A potential with respect to the Fubini-Study metric $\omega_{FS}$ for the generator of the induced circle action is given by
\begin{equation*}
u_{FS} = \frac{\alpha_0 |x_0|^2 + \alpha_1 |x_1|^2 + \alpha_2 |x_2|^2}{|x|^2}.
\end{equation*}
By direct calculation, one can check that the average $\underline u = \int_X u_{FS} \omega_{FS}^3/ \int_X \omega_{FS}^3$ is zero.
Now consider a resolution $\pi : M \to X$ and let, as in the general case discussed above, $E_j$ be the exceptional divisor over $p_j \in S$.
Chosen an integer $b_j>0$ for each $p_j \in S$, consider the line bundle $$L_r = \pi^*L^r \otimes \mathcal O(-\sum_{j=0}^2 b_jE_j),$$ which is ample for all $r$ sufficiently large.
By corollary \ref{cor::mainresultintersection} we get
\begin{equation*}
F(M,L_r)
= -\frac{1}{2r^2} \sum_{j=0}^2 \alpha_j K_M \cdot (b_jE_j)^{2} + O(1/r^3),
\end{equation*}
where we used that $F(X,L)=0$ thanks to K-polystability of $X$, that $\underline u=0$ as discussed above, and that $L^3=3$.
As a consequence, as soon as $b_j$ are chosen so that $K_M \cdot (b_0E_0)^{2}, K_M \cdot (b_1E_1)^{2},K_M \cdot (b_2E_2)^{2}$ are not all the same, one can choose the $\alpha_j$'s so that $F(M,L_r)$ is non-zero for large $r$.
Therefore we proved the following
\begin{prop}
\label{FDelta}
With the notation above, any polarized (log) resolution $(M,L_r)$ of the cubic threefold $F_\Delta=0$ is K-unstable for $r$ sufficiently large as soon as the intersection numbers $K_M \cdot (b_0E_0)^{2}, K_M \cdot (b_1E_1)^{2},K_M \cdot (b_2E_2)^{2}$ are not all the same.
\end{prop}

\subsection{$F_{A,B}$}
Let $X \subset \mathbf{CP}^4$ be the zero locus of $F_{A,B} = Ax_2^3 + x_0x_3^2 + x_1^2x_4 - x_0x_2x_4 + Bx_1x_2x_3$ where at least one of $A$ and $B$ is non-zero.
As described by Allcock \cite{All2003}, different choices of the pair $A$, $B$ give projectively equivalent threefolds if and only if they give the same $\beta = 4A/B^2 \in \mathbf C \cup \{\infty\}$.
In other words, $\beta$ is a moduli parameter.
The singularities of $X$ depend on $\beta$.
If $\beta \neq 0,1$ then $X$ has precisely two singular points of type $A_5$.
If $\beta = 0$ then an additional singular point of type $A_1$ appears.
If $\beta = 1$ then the singular locus of $X$ is a rational curve.
We drop the latter case since singularities are non-isolated.
On the other hand, remaining cases are quite similar each other. 
Therefore we consider in some detail the case $\beta=0$ and we leave the other ones as an exercise for the reader.

Thus, from now on, $X \subset \mathbf{CP}^4$ will be the zero locus of $F_{0,1} = x_0x_3^2 + x_1^2x_4 - x_0x_2x_4 + x_1x_2x_3$.
One can directly check that the singular locus of $X$ is constituted by three coordinate points 
\begin{equation}
S = \left\{ p_0=(1:0:0:0:0), p_2=(0:0:1:0:0), p_4=(0:0:0:0:1) \right\}.
\end{equation}
The points $p_0$, $p_4$ turn out to be singularities of type $A_5$, whereas $p_2$ is an $A_1$ singularity.
Looking for $\mathbf C^*$-actions on $\mathbf{CP}^4$ which preserve $X$, one find that all of them are coverings of the one induecd by $\diag(t^{-2},t^{-1},1,t,t^2)$, where $t \in \mathbf C^*$.
A potential with respect to the Fubini-Study metric $\omega_{FS}$ for the generator of the induced circle action is given by
\begin{equation*}
u_{FS} = \frac{-2 |x_0|^2 - |x_1|^2 + |x_3|^2 + 2 |x_4|^2}{|x|^2}.
\end{equation*}
Note that the transformation which maps $(x_0:\dots:x_4)$ to $(x_4:\dots:x_0)$ is an holomorphic isometry of $\mathbf{CP}^4$ that preserves $X$ and transforms $u_{FS}$ into $-u_{FS}$. 
As a consequence, the average $\underline u = \int_X u_{FS} \omega_{FS}^3/ \int_X \omega_{FS}^3$ is zero.
Now let $\pi : M \to X$ be a (log) resolution and let $E_j$ be the exceptional divisor over $p_j \in S$.
Choose an integer $b_j>0$ for each $p_j \in S$, and consider the line bundle $$L_r = \pi^*L^r \otimes \mathcal O(-\sum_{j=0}^2 b_{2j}E_{2j}),$$ which is ample for all $r$ sufficiently large.
By corollary \ref{cor::mainresultintersection} we get
\begin{equation}
F(M,L_r)
= \frac{1}{r^2} \sum_{j=0}^2 (1-j) K_M \cdot (b_{2j}E_{2j})^{2} + O(1/r^3),
\end{equation}
where we used that $F(X,L)=0$ thanks to K-polystability of $X$, that $\underline u=0$ as discussed above, and that $L^3=3$.
Note that the local resolution chosen for the $A_1$ singularity $p_2$ does not affect the stability of $(M,L_r)$.
On the other hand, $F(M,L_r)$ is non-zero for all $r$ sufficiently large whenever $b_0$, $b_4$ are chosen so that $K_M \cdot (b_0E_0)^{2} + K_M \cdot (b_4E_4)^{2} \neq 0$.

A minor adjustment of argument above extends the result above for resolutions of the zero locus of $F_{A,B}$ with $B^2\neq4A$.
Summarizing we have the following
\begin{prop}
With notation above, any polarized (log) resolution $(M,L_r)$ of the cubic threefold $F_{A,B}=0$ with $4A\neq B^2$ is K-unstable for $r$ sufficiently large as soon as $K_M \cdot (b_0E_0)^{2} + K_M \cdot (b_4E_4)^{2} \neq 0$.
\end{prop}


\begin{thebibliography}{99}

\bibitem{All2003} D. Allcock.
		\emph{``The moduli space of cubic threefolds''}.
		J. Algebraic Geom. {\bf 12} (2003), no. 2, 201--223. 
		
\bibitem{ADVLM} C. Arezzo, A. Della Vedova, R. Lena and L. Mazzieri.
		\emph{``On the Kummer construction for cscK metrics"}. 
		Boll Unione Mat Ital (2018). https://doi.org/10.1007/s40574-018-0170-4.
		
\bibitem{AreDelMaz2018} C. Arezzo, A. Della Vedova and L. Mazzieri. 
		\emph{``K-stability, Futaki invariants and cscK metrics on orbifold resolutions''}.
		arXiv:1808.08420.

\bibitem{AreDelLaN2012} C. Arezzo, A. Della Vedova and  G. La Nave.
		\emph{``Singularities and K-semistability''}.
		Int. Math. Res. Not. IMRN 2012, no. 4, 849--869. 

\bibitem{AreDelLaN2009} C. Arezzo, A. Della Vedova and  G. La Nave.
		\emph{``K-destabilizing test configurations with smooth central fiber''}.
		Variational problems in differential geometry, 24--36, London Math. Soc. Lecture Note Ser., 394, Cambridge Univ. Press, Cambridge, 2012. 

\bibitem{ArezzoLenaMazzieri2015} C. Arezzo, R. Lena and L. Mazzieri.
		\emph{``On the resolution of extremal and constant scalar curvature K\"ahler orbifold''}.
		Int. Math. Res. Not. IMRN 2016, no. 21, 6415--6452.

\bibitem{ap} C. Arezzo and F. Pacard.
		\emph{``Blowing up and desingularizing constant scalar curvature K\"ahler manifolds''}.
		Acta Math. {\bf 196} (2006), no. 2, 179--228. 

\bibitem{ap2} C. Arezzo and F. Pacard. 
		\emph{``Blowing up K\"ahler manifolds with constant scalar curvature. II''}.
		Ann. of Math. (2) {\bf 170} (2009), no. 2, 685--738. 
  
\bibitem{aps} C. Arezzo, F. Pacard, and M. Singer.
		\emph{``Extremal metrics on blowups''}.
		Duke Math. J. \textbf{157} (2011), no. 1, 1--51.

\bibitem{Ban2006} S. Bando.
		\emph{``An obstruction for Chern class forms to be harmonic''}.
		Kodai Math. J. {\bf 29} (2006), no. 3, 337--345.

\bibitem{BerGetVer2004} N. Berline, E. Getzler and M. Vergne.
		\emph{``Heat kernels and Dirac operators''}.
		Corrected reprint of the 1992 original. Grundlehren Text Editions. Springer-Verlag, Berlin, 2004.
		
\bibitem{br} O. Biquard and Y. Rollin. 
		\emph{``Smoothing singular constant scalar curvature K\"ahler surfaces and
minimal Lagrangians''}.
		Adv. Math. {\bf 285} (2015), 980--1024.

\bibitem{Cal1985} E. Calabi.
		\emph{``Extremal K\"ahler metrics, II''}.
		In Differential Geometry and Complex Analysis, 
		eds. I. Chavel and H. M. Farkas, Springer Verlag (1985), 95--114.

\bibitem{Del2008} A. Della Vedova.
		\emph{``CM-stability of blow-ups and canonical metric''}.
		arXiv:0810.5584.

		
\bibitem{DinTia1992} W. Ding and G. Tian.
		\emph{``K\"ahler-Einstein metrics and the generalized Futaki invariants''}.
		Invent. Math., \textbf{110} (1992), 315--335.
		
\bibitem{Dol2003} I. Dolgachev.
		\emph{``Lectures on invariant theory''}.
		London Mathematical Society Lecture Note Series, 296. Cambridge University Press, Cambridge, 2003. 
		
\bibitem{Don2002} S. K. Donaldson.
		\emph{``Scalar curvature and stability of toric varieties''}.
		J. Differential Geom. \textbf{62} (2002), no. 2, 289--349.
		
\bibitem{Don2005} S. K. Donaldson.
		\emph{``Lower bounds on the Calabi functional''}. 
		J. Differential Geom. \textbf{70} (2005), no. 3, 453--472.
		
\bibitem{Fut1983} A. Futaki.		
		\emph{``An obstruction to the existence of Einstein K\"ahler metrics''}, 
		Invent. Math. {\bf 73} (1983), 437--443

\bibitem{Fut1988} A. Futaki.
		\emph{``K\"ahler-Einstein metrics and Integral Invariants''}.
		Springer-LNM 1314, Springer-Verlag (1988).

\bibitem{Ish1987} S. Ishii.
		\emph{``Isolated Q-Gorenstein singularities of dimension three''}. 
		Complex analytic singularities, 165--198, Adv. Stud. Pure Math., 8, North-Holland, Amsterdam, 1987.

\bibitem{Laz2004} R. Lazarsfeld.
		\emph{``Positivity in algebraic geometry - I''} 
		Ergebnisse der Mathematik und ihrer Grenzgebiete, 48. Springer-Verlag, Berlin, 2004.

\bibitem{lx} C. Li and C. Xu.
		\emph{``Special test configuration and K-stability of Fano varieties''}.
		Ann. of Math. (2) {\bf 180} (2014), no. 1, 197--232. 

\bibitem{LiuXu2017} Y. Liu and C. Xu.
		\emph{``K-stability of cubic threefolds''}.
		arXiv:1706.01933v2.

\bibitem{pt} S. Paul and G. Tian.
		\emph{``CM Stability and the Generalized Futaki Invariant I''}.
		arXiv:math/0605278.
		
\bibitem{sp} C. Spotti.
		\emph{``Deformations of nodal K\"ahler-Einstein del Pezzo surfaces with discrete automorphism groups''}. 
		J. Lond. Math. Soc. (2), {\bf 89} 539-558, 2014.

\bibitem{Sto2010} J. Stoppa.
		\emph{``Unstable blowups"}. 
		J. Algebraic Geom. {\bf 19} (2010), no. 1, 1--17.
		
\bibitem{sz} G. Sz{\'e}kelyhidi.
		\emph{``On blowing up extremal K\"ahler manifolds''}.
		Duke Math. J. \textbf{161} (2012), no. 8, 1411--1453. 

\bibitem{sz2} G. Sz{\'e}kelyhidi.
		\emph{``On blowing up extremal K\"ahler manifolds II''}.
		Invent. Math. \textbf{200} (2015), no. 3, 925--977.
	
\bibitem{szy} R. Seyyedali and G. Sz{\'e}kelyhidi.
		\emph{``Extremal metrics on blowups along submanifolds''}.
		arXiv:1610.06865.

\bibitem{Tia1997} G. Tian.
                \emph{K\"ahler-Einstein metrics with positive scalar curvature}.
                Invent. Math. {\bf 130} (1997), no. 1, 1--37.
                
\bibitem{tiCheeger} G. Tian.
		\emph{``Existence of Einstein metrics on Fano manifolds''}. 
		Metric and differential geometry, 119--159, 
		Progr. Math., {\bf 297}, Birkh\"auser/Springer, Basel, 2012.

\end{thebibliography}
\end{document}